\documentclass[a4paper]{amsart}
\usepackage{tikz}
\usepackage[alphabetic,initials]{amsrefs}
\usepackage{url}
\usepackage{paralist}

\theoremstyle{plain}
\newtheorem{thm}{Theorem}[section]
\newtheorem{lem}[thm]{Lemma}

\theoremstyle{definition}
\newtheorem{define}[thm]{Definition}

\theoremstyle{remark}
\newtheorem*{rem}{Remark}

\def \t {\mathbf t}
\def \x {\mathbf x}
\def \y {\mathbf y}
\def \o {\mathbf o}
\def \c {\mathbf c}

\begin{document}

\title{Distance Geometry for Kissing Spheres}
\author{Hao Chen}
\address{Freie Universit\"at Berlin, Institut f\"ur Mathematik}
\begin{abstract}
  A kissing sphere is a sphere that is tangent to a fixed reference ball.
  We develop in this paper a distance geometry for kissing spheres,
  which turns out to be a generalization of the classical Euclidean distance geometry.
\end{abstract}
\thanks{
This research was supported by the Deutsche Forschungsgemeinschaft 
within the research training group `Methods for Discrete Structures' (GRK 1408).
}
\maketitle

\section{Introduction}\label{sec:Intro}
Distance geometry studies the geometry based only on knowledge of distances.

We develop in this paper a distance geometry for kissing spheres following the approach of Euclidean distance geometry.
We first establish a distance space by defining a distance function (Section \ref{sec:Dist}) on the set of kissing spheres.
Then we study two basic problems of distance geometry: 
the embeddability problem (Section \ref{sec:Embed}) and the distance completion problem (Section \ref{sec:Compl}). 

The key observation of this paper is that
the distance matrix for kissing spheres also plays the role of Cayley-Menger matrix.
It is then possible to adapt the proof techniques from Euclidean distance geometry for our use.
Our main results (Theorem \ref{EmbK} and \ref{CompK}) 
are similar to the results in Euclidean distance geometry (Theorem \ref{EmbE} and \ref{CompE}).
We also notice that the distance geometry for kissing spheres may degenerate to Euclidean distance geometry in different ways.
In this sense, the distance geometry developped in this paper generalizes Euclidean distance geometry.

At the end of the paper,
we will introduce some previous works on spheres,
and point out the similarity and relations to our results.
\subsection{Euclidean distance geometry}\label{sec:Euc}
Let $X$ be a set. 
A non-negative symmetric function $d:X\times X\to\mathbb{R}_{\geq 0}$ is called a \emph{distance function} on $X$. 
The pair $(X,d)$ is called a \emph{distance space}.
A first example would be the Euclidean distance space $(\mathbb{E}^n,d_E)$.

Euclidean distance geometry studies the geometry of points only with knowledge of Euclidean distances.
It answers questions like this: 
Is it possible to find three points $A$, $B$, $C$ in a plane 
such that the distances between them are $d_E(AB)=3$, $d_E(BC)=4$ and $d_E(CA)=5$?

In the language of distance geometry, this is an embeddability problem:
\begin{define}[Isometric Embedding]
  Let $(I,d)$ and $(I',d')$ be two distance spaces.
  We say that $(I,d)$ is \emph{isometrically embeddable} into $(I',d')$,
  if there exists an \emph{isometric embedding} $\sigma:I\to I'$, 
  such that $d'(\sigma(x),\sigma(y))=d(x,y)$ for all $x,y\in I$.
\end{define}
The set $I$ is often finite in distance geometry. 
In this case, we label the elements of $I$ by integers $0,\ldots,k$ where $|I|=k+1$,
and write $\sigma_i$ instead of $\sigma(i)$.
So the \emph{embeddability problem} asks: 
Given a finite distance space $(I,d)$, is it isometrically embeddable into the Euclidean distance space $(\mathbb{E}^n,d_E)$? 

There are two powerful tools for solving this problem.
One is the \emph{distance matrix} $D_I$, 
defined as a $(k+1)\times(k+1)$ matrix whose $i,j$ entry is the squared distance $d(i,j)^2$ for $i,j\in I$.
The other is the \emph{Cayley-Menger matrix} $M_I$, 
defined as 
$M_I=\bigl(\begin{smallmatrix} 
  D_I&e\\e^T&0
\end{smallmatrix}\bigr)$, 
where $e$ denotes the all-ones vector.

The following theorem combines some important results mentioned in \cites{menger1954, hayden1988, graham1985, gower1985}.
\begin{thm}\label{EmbE}
  Let $(I,d)$ be a distance space.
  Consider the following statements
  \begin{enumerate}[\rm 1)]
      \renewcommand{\theenumi}{\rm (\roman{enumi})}
      \renewcommand{\labelenumi}{\theenumi}
    \item \label{euc1} $(I,d)$ is isometrically embeddable into $(\mathbb{E}^n,d_E)$.
    \item \label{euc2} $(-1)^{|J|}\det M_J\geq 0$ for all $J\subseteq I$, and the rank of $M_I$ is at most $n+2$.
    \item \label{euc3} $M_I$ has exactly one positive eigenvalue, and at most $n+1$ negative eigenvalues.
    \item \label{euc4} $D_I$ has exactly one positive eigenvalue, and at most $n+1$ negative eigenvalues.
  \end{enumerate}
  Then, $\ref{euc1}\Leftrightarrow\ref{euc2}\Leftrightarrow\ref{euc3}\Rightarrow\ref{euc4}$
\end{thm}
The Euclidean distance matrix also provides information on the cosphericity of the points,
we refer to \cite{gower1985} for more details.

If we are not given a complete information about the distances, the right question to ask is the \emph{distance completion problem}: 
can we complete the unknown distances so that the completed distance space is embeddable into $(\mathbb{E}^n,d_E)$?
This problem can be formulated as follows in the language of graph theory:

\begin{define}[Distance completion]
  Given an undirected graph $G=(V,E)$ and a length function $\ell:E\to \mathbb{R}_{\geq 0}$,
  we say that $(G,\ell)$ is \emph{completable} in a distance space $(I,d)$ 
  if there is a pre-metric $d_V$ on the vertex set $V$,
  such that $(V,d_V)$ is isometrically embeddable into $(I,d)$, and $d(u,v)=\ell(u,v)$ for all $(u,v)\in E$.
\end{define}

So the Euclidean distance completion problem asks:
Given a graph $G$ and a length function $\ell$ on $G$, is $(G,\ell)$ completable in $(\mathbb{E}^n,d_E)$?

We define two sets of length functions as follows
\begin{align*}
  \mathcal{C}_E^n(G)&=\{\ell:E\to\mathbb{R}_{\geq 0}\mid(G,\ell)\text{ is completable in } (\mathbb{E}^n,d_E)\}\\
  \mathcal{K}_E^n(G)&=\{\ell:E\to\mathbb{R}_{\geq 0}\mid\text{for all cliques }K\in G, (K,\ell)\text{ is completable in } (\mathbb{E}^n,d_E)\}
\end{align*}
In general they are not equal, but we have the following theorem
\begin{thm}[\cite{laurent1998}]\label{CompE}
  $\mathcal{C}_E^n(G)=\mathcal{K}_E^n(G)$ if and only if $G$ is chordal.
\end{thm}
Here, a \emph{chordal graph} is a graph without chordless cycles longer than $3$,
or equivalently, every cycle longer than $3$ has an edge joining two vertices not adjacent in the cycle.

Therefore, for a choral graph, in order to tell if a length function is completable or not, we only need to check all the cliques.

\subsection{Kissing spheres}\label{sec:KSphere}
We now define kissing spheres, the main object of study of this paper.

We work in the extended $n$-dimensional Euclidean space $\hat{\mathbb{E}}^n=\mathbb{E}^n\cup\{\infty\}$, 
with Cartesian coordinate system $(x_0,\dots,x_{n-1})$.
A sphere centered at $\o\in\hat{\mathbb{E}}^n$ with diameter $\phi$ is the set
$\{\x\mid d_E(\x,\o)=\phi/2\}$.
We consider $(n-1)$-hyperplanes as spheres of infinite diameter.

Let $\kappa$ be a real number, a \emph{ball} of curvature $\kappa$ may be one of the followings:
\begin{inparaenum}[i)]
\item A set $\{\x\mid d_E(\x,\o)\leq 1/\kappa\}$ if $\kappa>0$;
\item A set $\{\x\mid d_E(\x,\o)\geq-1/\kappa\}$ if $\kappa<0$;
\item A closed half-spaces if $\kappa=0$.
\end{inparaenum}
In the first two cases, $\o\in\mathbb{E}^n$ is the center of the ball.

A sphere is said to be \emph{tangent} to a ball at a point $\t\in\mathbb{E}^n$,
if $\t$ is the only element in their intersection.
We call $\t$ the \emph{tangent point}, 
which can be at infinity if it involves a ball of $0$ curvature and a sphere of infinit diameter.

\begin{define}[Kissing sphere]\label{Def1}
  Fix a ball in $\hat{\mathbb{E}}^n$ as the reference ball,
  a \emph{kissing sphere} is a sphere tangent to the reference ball.
\end{define}

Our main concern is the combinatorics, 
i.e. relations like tangency, intersection, or disjointness between the kissing spheres.
These are defined as follows:
two kissing spheres are \emph{tangent} to each other if their intersection consists of a single point
which is not on the boundary of the reference ball;
two kissing spheres \emph{intersect} if their intersections consists of more than one point,
or if they are tangent to the reference ball at a same tangent point;
the \emph{disjointness} is defined as usual.

Without loss of generality, we can assume the reference ball to be the half-space $x_0\leq 0$.
If it is not the case, a M\"obius transformation can send the reference ball to the half-space without any change to the combinatorics.

Therefore we define a kissing sphere alternatively as follows
\begin{define}[Kissing sphere]\label{Def2}
  A kissing sphere in $\hat{\mathbb{E}}^n$ is a sphere in the half-space $x_0\geq 0$ tangent to the hyperplane $x_0=0$.
\end{define}
Note that the tangent point can be at infinity.
In this case, the kissing sphere is a hyperplane $x_0=h>0$, 
and we say that it is a hyperplane at distance $h$. 

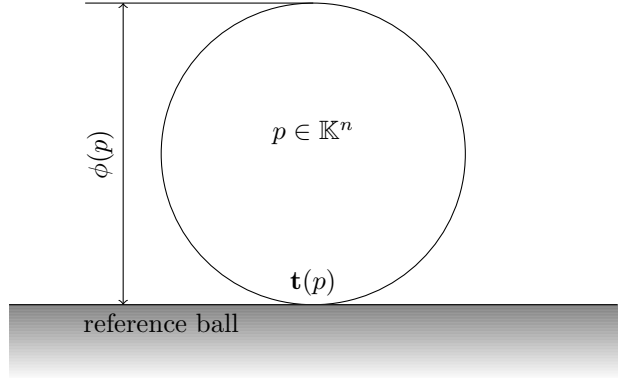
\begin{figure}[htb]
  \begin{center}
    \begin{tikzpicture}
      \draw (0,2) circle (2cm) 
      node[above] {$p\in\mathbb{K}^n$};
      \shade (-4,0) rectangle (4,-1);
      \draw (-4,0) -- (4,0)
      node[below,near start] {reference ball}
      node[above,midway] {$\t(p)$};
      \draw (-3,4) -- (0,4);
      \draw[<->] (-2.5,0) -- (-2.5,4)
      node[sloped,midway,above] {$\phi(p)$};
    \end{tikzpicture}
  \end{center}
  \label{DefKB}
  \caption{A kissing sphere as defined in Definition \ref{Def2}.}
\end{figure}

We denote by $\mathbb{K}^n$ the set of kissing spheres in $\hat{\mathbb{E}}^n$ as defined in Definition \ref{Def2}.
For a kissing sphere $p\in\mathbb{K}^n$, as shown in Figure \ref{DefKB} for $n=2$,  
we denote by $\t(p)\in\hat{\mathbb{E}}^{n-1}$ the tangent point on the $(n-1)$-dimensional hyperplane $x_0=0$, 
and by $\phi(p)\in\mathbb{R}\cup\{\infty\}$ the diameter of $p$.
The pair $(\phi(p),\t(p))$ is then the ``north pole'' of $p$, situated in the half-space $x_0>0$.

\section{Distance function}\label{sec:Dist}
\emph{M\"obius transformations} on $\hat{\mathbb{E}}^n$ are diffeomorphisms
\footnote{
As in \cite{hertrich-jeromin2003}, we don't require a M\"obius transformation to preserve the orientation,
therefore reflections and inversions are also M\"obius transformations.
}
that map spheres to spheres.
They form a group called the \emph{M\"obius group}, denoted by $\text{M\"ob}(n)$. 

A M\"obius transformation $T\in\text{M\"ob}(n)$ that preserves the half-space $x_0\leq 0$ maps kissing spheres to kissing spheres.
It also maps between spheres centered on the hyperplane $x_0=0$,
therefore the restriction of $T$ on the hyperplane $x_0=0$ is a M\"obius transformation on $\hat{\mathbb{E}}^{n-1}$.
Conversely, by \emph{Poincar\'e extension} \cite{beardon1983}*{Section 3.3},
each M\"obius transformation on $\hat{\mathbb{E}}^{n-1}$ 
can be naturally extended to a M\"obius transformation on $\hat{\mathbb{E}}^n$ that preserves the half-space $x_0\leq 0$.
We thus define the action of $T\in\text{M\"ob}(n-1)$ on $\mathbb{K}^n$ to be the action of its Poincar\'e extension.

We define a distance $d_K$ on $\mathbb{K}^n$ as follows:
\begin{define}[Distance for kissing spheres]
  Let $p,q$ be two elements of $\mathbb{K}^n$.
  Let $T\in\text{M\"ob}(n-1)$ be, if there exists, a M\"obius transformation preserving the half-space $x_0\leq 0$, such that
  $$
  \phi(Tp)=\phi(Tq)=1
  $$
  Then the distance between $p$ and $q$ is $d_K(p,q):=d_E(\t(Tp),\t(Tq))$.
  If such a $T$ does not exist, $d_K(p,q):=0$.
\end{define}

\begin{figure}[htb]
  \begin{center}
    \begin{tikzpicture}
      \draw (-6,0) -- (-2,0);
      \draw (-4.5,0) +(0,1.5) circle (1.5cm) node {$p$};
      \draw (-2.5,0) +(0,0.5) circle (0.5cm) node {$q$};
      \draw[thick,->] (-1,1) -- (1,1) node[midway, above] {$T$};
      \draw (1.5,0.5) -- (6,0.5);
      \draw (2.5,0.5) +(0,1) circle (1cm) node {$Tp$};
      \draw (5,0.5) +(0,1) circle (1cm) node {$Tq$};
      \draw[|<->|] (2.5,0.3) -- (5,0.3) node[midway, below] {$d_K(p,q)$};
    \end{tikzpicture}
  \end{center}
  \caption{Definition of the distance, if $T$ exists.}
\end{figure}
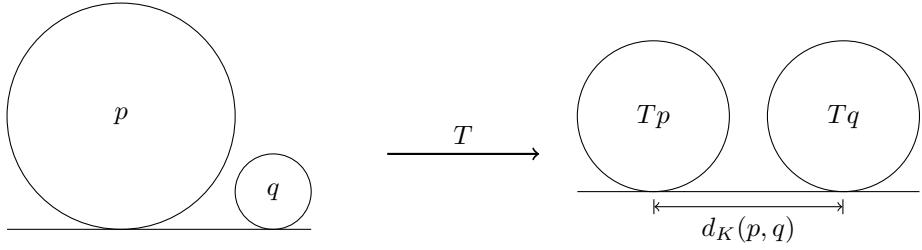
\begin{thm}\label{distance}
  $d_K$ is well defined, independent of the choice of $T$. 
\end{thm}

That is, $d_K$ defined on $\mathbb{K}^n$ is invariant under the action of $\text{M\"ob}(n-1)$.

As a warm-up before the proof, we shall look at the effect of an inversion preserving the half-space $x_0\leq 0$ on a kissing sphere.
Let $s$ be a sphere centered at a point $\o$ on the hyperplane $x_0=0$ with radius $r$.
For a kissing sphere $p\in\mathbb{K}^n$, we denote by $p^s$ the image of $p$ under the inversion transform with respect to $s$.
We have
\begin{align}
  \phi(p^s)&=\frac{r^2\phi(p)}{d_E(\o,\t(p))^2}\label{Invphi}\\
  d_E(\o,\t(p^s))&=\frac{r^2}{d_E(\o,\t(p))}\label{Invt}
\end{align}
where $d_E(\x,\y)$ is the Euclidean distance on the hyperplane $x_0=0$.

The effect of such an inversion on $p$ is then the same as the effect of a dilation of scale factor $r^2/d_E(\o,\t(p))^2$.
The scale factor does not depend on the diameter of $p$.
This will be a useful fact.

\begin{proof}[Proof of Theorem \ref{distance}]
  An explicit calculation is not necessary, but may help understanding the situation.

  Let $p,q$ be two kissing spheres such that $\phi(p),\phi(q)<\infty$ and $d_E(\t(p),\t(q))>0$. 
  The infinite case and the degenerate case will be discussed later.

  Choose a point $\o$ on the line segment $\t(p)\t(q)$, such that $d_E(\o,\t(p))/d_E(\o,\t(q))=\sqrt{\phi(p)/\phi(q)}$, i.e. 
  \begin{align*}
    d_E(\o,\t(p))&=\frac{d_E(\t(p),\t(q))\sqrt{\phi(p)}}{\sqrt{\phi(p)}+\sqrt{\phi(q)}}\\
    d_E(\o,\t(q))&=\frac{d_E(\t(p),\t(q))\sqrt{\phi(q)}}{\sqrt{\phi(p)}+\sqrt{\phi(q)}}
  \end{align*}
  Let $s$ be a sphere centered at $\o$ with radius 
  $$
  r=\frac{d_E(\t(p),\t(q))}{\sqrt{\phi(p)}+\sqrt{\phi(q)}}
  $$
  Then, by \eqref{Invphi} and \ref{Invt}, we have $\phi(p^s)=\phi(q^s)=1$, and 
  \begin{equation}\label{normal}
    \begin{split}
      d_K(p,q)&=d_E(\t(p^s),\t(q^s))=d_E(\o,\t(p^s))+d_E(\o,\t(q^s))\\
      &=\frac{d_E(\t(p),\t(q))}{\sqrt{\phi(p)\phi(q)}}
    \end{split}
  \end{equation}

  A M\"obius transformation is generated by reflections and inversions.
  For details, see \cite{beardon1983}*{Definition 3.1.1} where inversions are thought as reflections with respect to a sphere, 
  or \cite{cecil2008}*{Theorem 3.8} where reflections are thougth as inversions with respect to a plane.
  Theorem \ref{distance} is obviously true for reflections. We shall study the inversions in detail.

  An inversion preserving the half-space $x_0\leq 0$ must have its inversion sphere centered at a point $\o$ on the hyperplane $x_0=0$. 
  Let $s$ be such an inversion sphere of radius $r$. 
  By \eqref{Invt}, we have $d_E(\o,\t(p^s))=r^2/d_E(\o,\t(p))$ and $d_E(\o,\t(q^s))=r^2/d_E(\o,\t(q))$. 
  Thanks to the independence of the scale factor of the diameter, the triangle $\o\t(p)\t(q)$ and the triangle $\o\t(q^s)\t(p^s)$ are similar, and 
  $$
  d_E(\t(p^s),\t(q^s))=\frac{r^2d_E(\t(p),\t(q))}{d_E(\o,\t(p))d_E(\o,\t(q))}
  $$
  We then have
  $$
  d_K(p^s,q^s)=\frac{d_E(\t(p^s),\t(q^s))}{\sqrt{\phi(p^s)\phi(q^s)}}=\frac{d_E(\t(p),\t(q))}{\sqrt{\phi(p)\phi(q)}}=d_K(p,q)
  $$
  Which proves the theorem.

  We now extend the calculation to the infinite case. 
  Let $p$ be a hyperplane at distance $h$. 
  Consider again a sphere $s$ centered at a point $\o$ on $x_0=0$ with radius $r$. 
  Then $\phi(q^s)=r^2\phi(q)/d_E(\o,\t(q))^2$, $\phi(p^s)=r^2/h$, and $d_E(\t(p^s),\t(q^s))=r^2/d_E(\o,\t(q))$.
  Since the inversion preserves the distance $d_K$, by \eqref{normal}, we have
  \begin{equation}\label{infinite}
    d_K(p,q)=\frac{d_E(\t(p^s),\t(q^s))}{\sqrt{\phi(p^s)\phi(q^s)}}=\sqrt{\frac{h}{\phi(q)}}
  \end{equation}

  Finally we study the degenerate case, i.e. $\t(p)=\t(q)$ but $p\neq q$.

  Since a M\"obius transformation is bijective, it is impossible to transform $p$ and $q$ into spheres of same diameter.
  According to the definition, $d_K(p,q)=0$.
  This is reasonable, since it's the limit of \eqref{normal} as $d_E(\t(p),\t(q))$ tends to $0$,
  or the limit of \eqref{infinite} as $\phi(q)$ tends to infinity.
  This is M\"obius invariant since $\t(p)=\t(q)$ holds under any M\"obius transformation.
\end{proof}

\begin{rem}
  $d_K$ is not a metric, but a \emph{pre-metric}, 
  i.e. for all $p,q\in\mathbb{K}^n$, we have
  \begin{inparaenum}[i)]
  \item $d_K(p,p)=0$;
  \item $d_K(p,q)\geq 0$ (non-negativity); and
  \item $d_K(p,q)=d_K(q,p)$ (symmetry).
  \end{inparaenum}
  However, the triangle inequality may not be satisfied, 
  and there may be $p,q\in\mathbb{K}^n$ such that $p\neq q$ but $d_K(p,q)=0$.
\end{rem}

We notice that $d_K$ reflects the combinatorics.
More specifically, $d_K=1$ if two kissing spheres are tangent to each other, 
$>1$ if they are disjoint, $<1$ if they intersect, 
and $=0$ if they are tangent to the reference ball at a same point. 

We notice from \eqref{normal} that 
the distance space for a set of kissing spheres is discretely conformally Euclidean:
\begin{define}[discrete conformal equivalence \cites{luo2004,bobenko2010}]

  Two discrete distance spaces $(I,d)$ and $(I',d')$ are \emph{conformally equivalent},
  if there exists a mapping $\xi:I\to I'$ and a real valued function $f:I\to\mathbb{R}_{\geq 0}$, 
  such that $d'(\xi_i,\xi_j)=f(i)f(j)d(i,j)$ for all $i,j\in I$.
  We say that $\xi$ is a \emph{conformal mapping} with the \emph{conformal factor} $f$.

  $(I,d)$ is \emph{conformally Euclidean}, if it is conformally equivalent to Euclidean distance space.
\end{define}
In fact, if $(I,d)$ is embeddable into $(\mathbb{K}^n,d_K)$,
we can choose an isometric embedding $\sigma$ such that $\phi(\sigma_i)<\infty$ for all $i\in I$.
We recognise in \eqref{normal} that $\t\circ\sigma$ conformally maps $(I,d)$ to $(\mathbb{E}^{n-1},d_E)$, 
with conformal factor $\sqrt{\phi\circ\sigma}$. 
We also notice that, if $\phi(\sigma_i)=1$ for all $i\in I$, $d_K$ degenerates to Euclidean distance.

\section{Embeddability problem}\label{sec:Embed}
We now prove our first main theorem,
\begin{thm}\label{EmbK}
  Let $(I,d)$ be a finite distance space.
  The following statements are equivalent
  \begin{enumerate}
      \renewcommand{\theenumi}{\rm (\roman{enumi})}
      \renewcommand{\labelenumi}{\theenumi}
    \item\label{kiss1} $(I,d)$ is isometrically embeddable into $(\mathbb{K}^n,d_K)$.
    \item\label{kiss2} $(-1)^{|J|}\det D_J\leq 0$ for all $J\subseteq I$, and the rank of $D_I$ is at most $n+1$.
    \item\label{kiss3} $D_I$ has exactly one positive eigenvalue, and at most $n$ negative eigenvalues.
  \end{enumerate}
\end{thm}

Our proof is inspired by the proofs in \cite{deza1997}*{Sect. 6.2}.
It will use the notion of \emph{Schur complement}:
Consider a block matrix 
$\bigl(\begin{smallmatrix} 
  A&B\\C&D
\end{smallmatrix}\bigr)$, 
where $A$ and $D$ are square matrices.
After a block Gaussian elimination, it becomes 
$\bigl(\begin{smallmatrix} 
  P&0\\0&D 
\end{smallmatrix}\bigr)$, 
where $P=A-BD^{-1}C$ is called the Schur complement of $D$.

\begin{proof}
  Let $|I|=k+1$ and label the elements of $I$ by $0,\ldots,k$.
  If $(I,d)$ is isometrically embeddable into $(\mathbb{K}^n,d_K)$, 
  we can choose an embedding $\sigma:I\to\mathbb{K}^n$ such that $\sigma_k$ is the hyperplane at distance $1$.
  This is always possible, because $d_K$ is invariant under the action of $\text{M\"ob}(n-1)$.
  We then write the distance matrix $D_I$ explicitly.
  It will be in the form
  \begin{equation}\label{CMKiss}
    D_I=\begin{pmatrix}	
      D_{I\setminus\{k\}} & \Phi \\ 
      \Phi^t & 0 
    \end{pmatrix}
  \end{equation}
  where $\Phi$ denotes a $k\times 1$ column matrix whose $i$-th entry is $1/\phi(\sigma_i)$ for $0\leq i<k$.

  Consider the sub-distance-matrix $D_{\{0,k\}}$, one can easily verify that its Schur complement, denoted by $P_I$, is in the form
  \begin{equation*}
    \begin{split}
      (P_I)_{ij}&=\frac{d_E(\t(\sigma_i),\t(\sigma_j))^2-d_E(\t(\sigma_i),\t(\sigma_0))^2-d_E(\t(\sigma_0),\t(\sigma_j))^2}{\phi(\sigma_i)\phi(\sigma_j)}\\
      &=-2\left\langle\frac{\t(\sigma_i)-\t(\sigma_0)}{\phi(\sigma_i)},\frac{\t(\sigma_j)-\t(\sigma_0)}{\phi(\sigma_j)}\right\rangle
    \end{split}
  \end{equation*}
  for $i,j\in I\setminus\{0,k\}$, where $\langle\cdot,\cdot\rangle$ is the Euclidean inner product. 
  Since $I\setminus\{k\}$ is embedded by $\sigma$ into $\mathbb{K}^n$, 
  $\mathbf v_i=\t(\sigma_i)-\t(\sigma_0)$ are vectors in the $(n-1)$-dimensional hyperplane $x_0=0$.
  Therefore, $P_I$ is negative semi-definite, whose rank is at most $n-1$.

  Conversely, let $P_I$ be the Schur complement of the submatrix $D_{\{0,k\}}$.
  If it is negative semi-definite with rank at most $n-1$, it can be written in the form $(P_I)_{ij}=-2\langle\mathbf v_i, \mathbf v_j\rangle$ for $0<i,j<k$, 
  where $\mathbf v_i$ are vectors in an $(n-1)$-dimensional hyperplane, which can be, without loss of generality, assumed to be $x_0=0$.
  Then an embedding can be constructed by setting $\sigma_k$ as the hyperplane at distance $1$, and
  \begin{align*}
    \phi(\sigma_i)&=1/(D_I)_{ik}\\
    \t(\sigma_i)&=\begin{cases}
      \mathbf v_i/(D_I)_{ik} & \text{if } i>0\\
      0 & \text{if } i=0
    \end{cases}
  \end{align*}
  for $0\leq i<k$.

  We have proved that $(I,d)$ is isometrically embeddable into $(\mathbb{K}^n,d_K)$, if and only if $P_I$ is negative semi-definite of rank at most $n-1$.
  We also have the following relations for the Schur complement
  \begin{subequations}\label{schur}
    \begin{align}
      \det D_I&=-\frac{\det P_I}{\phi(\sigma_0)^2}\label{det}\\
      \text{In}D_I&=(1,1,0)+\text{In}P_I\label{In}\\
      \text{rk}D_I&=\text{rk}P_I+2\label{rk}
    \end{align}
  \end{subequations}
  where $\text{rk}M$ is the rank of $M$, and $\text{In}M$ is the inertia of matrix $M$, 
  which is a triple indicating (in order) the number of positive, negative and zero eigenvalues of $M$.
  Here we use the convention that the determinant of an empty matrix is $1$. 

  These relations allow us to express the negative semi-definiteness and the rank of $P_I$ by the distance matrices.

  A principal submatrix of $P_I$ is of the form $P_J$ for some subset $J$ of $I$ such that $J\supseteq\{0,k\}$.
  $P_I$ is negative semi-definite, if and only if for any principal submatrix $P_J$, $(-1)^{|J|}\det P_J\geq 0$.
  Notice that the choice of $\{0,k\}$ is totally arbitrary, since we can always apply a permutation on $I$, to bring any index to $0$ or $k$.
  So we have $(-1)^{|J|}\det P_J\geq 0$ for all $J\subseteq I$ (the case $|J|\leq 1$ is trivial).
  Then Relations \eqref{det} and \eqref{rk} prove $\ref{kiss1}\Leftrightarrow\ref{kiss2}$.

  $P_I$ is negative semi-definite, if and only if all its eigenvalues are nonpositive. 
  Then Relations \eqref{In} and \eqref{rk} prove $\ref{kiss1}\Leftrightarrow\ref{kiss3}$.
\end{proof}

Equation \eqref{CMKiss} shows that the distance matrix for kissing spheres is playing dual roles:
It combines the power of the distance matrix and the Cayley-Menger matrix.
Alternatively, this can be seen by comparing Theorem \ref{EmbE} and Theorem \ref{EmbK}, 

We also notice from \eqref{CMKiss} that if $\phi(\sigma_i)=1$ for all $i\in I\setminus\{k\}$, 
$\Phi=e$ and $D_I$ degenerates to an Euclidean Cayley-Menger matrix.

\section{Embedding into the lightcone}\label{sec:LCone}
We now show that the kissing spheres in $\hat{\mathbb{E}}^n$ can be embedded into the Minkowski space $\mathbb{R}^{n,1}$,
which will be useful later for the study of distance completion problem.

The \emph{Minkowski space} $\mathbb{R}^{n,1}$ is an $(n+1)$-dimensional vector space with an indefinite inner product of signature $(n,1)$.
Explicitely, with the coordinate system $\x=(x_0,\ldots,x_{n-1},t)$, 
the \emph{Minkowskian inner product} $\langle\cdot,\cdot\rangle_{n,1}$ is defined as 
$$
\langle\x,\x'\rangle_{n,1}=x_0x'_0+\ldots+x_{n-1}x'_{n-1}-tt'
$$
for two vectors $\x$ and $\x'$ in $\mathbb{R}^{n,1}$.
A vector $\x$ is \emph{space-like \textup{(}resp. null, time-like\textup{)}} if $\langle\x,\x\rangle_{n,1}$ is positive (resp. zero, negative).
The \emph{lightcone} is the set of null vectors,
i.e. $\mathbb{L}^n=\{\x\mid\langle\x,\x\rangle_{n,1}=0\}$.
A vector $\x$ is \emph{future- \textup{(}resp. past-\textup{)} directed} if $t$ is positive (resp. negative).

Let $(I,d)$ be a distance space isometrically embeddable into $(\mathbb{K}^n,d_K)$.
Since $D_I$ has exactly one positive eigenvalue and at most $n$ negative eigenvalues,
we can decompose it into $D_I=Q^t\Lambda Q$,
where $\Lambda$ is an $(n+1)\times(n+1)$ diagonal matrix, with $1$ as its last entry, and all other entries being $-1$.
The columns of $Q$ are vectors in the Minkowski space $\mathbb{R}^{n,1}$, indexed by the elements of $I$.
We can therefore write $(D_I)_{ij}=-\langle\x_i,\x_j\rangle_{n,1}$ 
for a system of vectors $\{\x_i\}_{i\in I}$ in $\mathbb{R}^{n,1}$. 

Since $\langle\x_i,\x_i\rangle_{n,1}=-(D_I)_{ii}=0$, all the vectors $\x_i$ are on the lightcone $\mathbb{L}^n$.

Since for all $i\neq j$ 
$$
-\langle\x_i,\x_j\rangle_{n,1}=\frac{1}{2}\langle\x_i-\x_j,\x_i-\x_j\rangle_{n,1}=(D_I)_{ij}\geq 0
$$
the difference between any two vectors can not be time-like.
Therefore, $\{\x_i\}_{i\in I}$ have to be either all future-directed, or all past-directed.
The Minkowskian inner product induces a pre-metric 
$$
d_M(\x,\x')^2:=-\langle\x,\x'\rangle_{n,1}
$$
on the future-directed lightcone $\mathbb{L}^n_+$.
We can therefore view $D_I$ as the Minkowskian distance matrix for a set of future-directed null vectors.

In fact, we just proved the following theorem:
\begin{thm}\label{lightcone}
  $(\mathbb{K}^n,d_K)$ is isometric to $(\mathbb{L}^n_+,d_M)$.
\end{thm}
Explicitly, let $\Psi$ be the isometry, then for a kissing sphere $p\in\mathbb{K}^n$ of finite diameter,
we have
\begin{equation}\label{Embed0}
  \Psi(p)=\frac{\sqrt{2}}{2\phi(p)}\left(1-\|\t(p)\|_2^2,2\t(p),1+\|\t(p)\|_2^2\right)
\end{equation}
It is easy to verify that $\langle\Psi(p),\Psi(p)\rangle_{n,1}=0$ and $d_K(p,q)=d_M(\Psi(p),\Psi(q))$.
We can extend $\Psi$ to kissing spheres of infinit diameter,
by setting $$\Psi(p)=\frac{\sqrt{2}}{2}(-h,0,\ldots,0,h)$$ if $p$ is a hyperplane at distance $h$.

We now make some remarks about this embedding: 

\begin{rem}
  If the reference ball has a non-zero curvature $\kappa$,
  we can either make it zero by a M\"obus transformation before applying the isometry in \eqref{Embed0}),
  or directly apply the following embedding to a kissing sphere $p$:
  \begin{equation}\label{refball}
    \Psi(p)=\left(\frac{\sqrt{2}}{2}+\frac{\sqrt{2}}{\kappa\phi(p)}\right)(\hat\t(p),1)
  \end{equation}
  where $\phi(p)$ is now the signed diameter, 
  negative if the sphere surrounds the reference ball,
  while $\hat\t(p)$ now means the unit direction vector of the tangent point,
  taking the center of the reference ball as the origin.
  This can be used to define a distance function for kissing spheres to a reference ball of non-zero curvature.
\end{rem}

\begin{rem}
  In the projective model of M\"obius geometry, points are mapped to null directions \cite{cecil2008}*{Equation 2.3}.
  In fact, our isometry maps the tangent points of kissing spheres to null directions in the same way,
  and use the vector lengths to distinguish different kissing spheres with a same tangent point. 
\end{rem}

\begin{rem}
  The invariance of $D_I$ under the action of $\text{M\"ob}(n-1)$ 
  is reflected in $\mathbb{R}^{n,1}$ as the invariance under Lorentz transformations that preserves the direction of time. 
  Indeed, $\text{M\"ob}(n-1)$ is isomorphic to the orthochronous Lorentz group $O_+(n,1)$ \cite{cecil2008}*{Corollary 3.3}.
\end{rem}

\begin{rem}
  A continuous version of Theorem \ref{lightcone} can be found in \cite{brinkmann1923}, 
  which states that a Riemann space is conformally Euclidean if and only if it can be embedded into the lightcone.  
  A generalisation for conformally flat Riemann manifolds can be found in \cite{asperti1989}
\end{rem}

We now show again that distance geometry for kissing spheres may degenerate to Euclidean distance geometry.

A normal vector $\y$ and a real number $c$ determines a hyperplane $H=\{\x\mid\langle\x,\y\rangle_{n,1}=c\}$.
A hyperplane is said to be \emph{space-like \textup{(}resp. null, time-like\textup{)}} if its normal vector is time-like (reps. null, space-like).

Let $I$ be a set of kissing spheres. 
If for all $p\in I$, $\Psi(p)$ lies on a same null-hyperplane $H$,
$H$ can be written in the form $H=\{\x\mid-\langle\x,\y\rangle_{n,1}=1\}$, where $\y$ is a null vector.
Therefore $\Psi^{-1}(\y)$ is tangent to all the elements of $I$.
We can find a Lorentz transformation $\mathcal{L}$ that sends $\y$ to $(-\frac{\sqrt{2}}{2},0,\cdots,0,\frac{\sqrt{2}}{2})$,
then for all $p\in I$, $\Psi^{-1}\mathcal{L}\Psi(p)$ is a unit kissing sphere (kissing spheres of unit diameter)
since they are tangent to the hyperplane at distance $1$ from the reference ball.
Therefore $D_I$ degenerates to Euclidean distance geometry,
as discussed in the end of Section \ref{sec:Embed}.

If for all $p\in I$, $\Psi(p)$ lies on a same space-like hyperplane $H$,
then we can find a Lorentz transformation $\mathcal{L}$ that sends $H$ to a time-constant hyperplane. 
The intersection of $\mathcal{L}H$ with the lightcone is an $(n-1)$-sphere, and $d_M$ degenerates to Euclidean distance $d_E$.
This is also observed in \cite{seidel1995}*{Theorem 4.5.3}.
One way to view this is to consider unit kissing spheres kissing a reference ball of \emph{finite} radius.
It turns out that $d_K$, induced by \eqref{refball}, equals the Euclidean distance between their centers.
\label{Lightcone}
\section{Distance completion problem}\label{sec:Compl}
We now study the distance completion problem in $(\mathbb{K}^n,d_K)$.
The results for embeddability problem tell us that
\begin{thm}\label{CompM}
  $(G,\ell)$ is completable in $(\mathbb{K}^n,d_K)$, 
  if and only if there is a non-negative symmetric matrix $D$ satisfying
  \begin{enumerate}
      \renewcommand{\theenumi}{\textbf{C\arabic{enumi}}}
      \renewcommand{\labelenumi}{\rm (\theenumi)}
    \item $D_{uv}=0$ if $u=v$. \label{comp1}
    \item $D_{uv}=\ell(u,v)^2$ if and only if $(u,v)\in E$. \label{comp2}
    \item the rank of $D$ is at most $n+1$. \label{comp3}
    \item $D$ has exactly one positive eigenvalue. \label{comp4}
  \end{enumerate}
\end{thm}
$D$ is in fact the distance matrix corresponding to a distance function realising the given edge lengths.
We call $D$ a \emph{target matrix}.

This theorem transforms the distance completion problem to the matrix completion problem:
some entries of the matrix being given (\ref{comp1} and \ref{comp2}), find the value for the other entries, 
so that the rank of matrix is low (\ref{comp3}).
We refer to \cites{johnson1990, candes2009} for more details about matrix completion.

Comparing to the classical matrix completion problem, \ref{comp4} is new.
The target matrix is usually positive semi-definite, but here we need it to be indefinit.
The result in the previous section can help here.

If $(G,\ell)$ is completable in $(\mathbb{K}^n,d_K)$, then for every clique $K$ of $G$, $(K,\ell)$ is completable in $(\mathbb{K}^n,d_K)$.
The inverse is in general not true.
Define two sets as in Euclidean case,
\begin{align*}
  \mathcal{C}_K^n(G)&=\{\ell:E\to\mathbb{R}_{\geq 0}\mid(G,\ell)\text{ is completable in }(\mathbb{K}^n,d_K)\}\\
  \mathcal{K}_K^n(G)&=\{\ell:E\to\mathbb{R}_{\geq 0}\mid\text{for all cliques }K\in G, (K,\ell)\text{ is completable in }(\mathbb{K}^n,d_K)\}
\end{align*}
We now prove our second main theorem 
\begin{thm}\label{CompK}
  $\mathcal{C}_K^n(G)=\mathcal{K}_K^n(G)$ if and only if $G$ is chordal.
\end{thm}

This is almost the same as Theorem \ref{CompE}. 
In fact, we have employed the proof techniques in \cite{laurent1998}, but with some necessary adaptions.

\begin{proof}[Proof of ``only if'']
  If $G$ is not chordal, consider a chordless cycle $C$ of length at least 4, and pick an edge $e_0\in C$.
  Construct a length function $\ell$ by setting $\ell(e)=1$ if $e$ has exactly one end in $C$ or if $e=e_0$, otherwise $\ell(e)=0$.
  Then $\ell\in\mathcal{K}_K^n(G)$ but $\ell\notin\mathcal{C}_K^n(G)$.
\end{proof}

If two graphs each has a clique of a same size, the \emph{clique-sum} glues them together by identifying that clique.
In the language of mathematics, consider a graph $G=(V,E)$ and two of its subgraphs $G_1=(V_1,E_1)$ and $G_2=(V_2,E_2)$.
$G$ is the clique-sum of $G_1$ and $G_2$, if $V=V_1\cup V_2$ and $W=V_1\cap V_2$ induces a clique in $G$, 
and there is no edge joining a vertex in $V_1\setminus W$ and a vertex in $V_2\setminus W$.

\begin{lem}\label{CliqueSum}
  Let $G$ be a clique-sum of $G_1(V_1,E_1)$ and $G_2(V_2,E_2)$. 
  If $\mathcal{C}_K^n(G_i)=\mathcal{K}_K^n(G_i)$ for $i=1,2$, then $\mathcal{C}_K^n(G)=\mathcal{K}_K^n(G)$.
\end{lem}

We use Theorem \ref{lightcone} to prove this lemma.
\begin{proof}[Proof of Lemma \ref{CliqueSum}]
  Let $\ell$ be an element in $\mathcal{K}_K^n(G)$.
  Obviously, $\ell\in\mathcal{K}_K^n(G_i)=\mathcal{C}_K^n(G_i)$, for $i=1,2$.
  Since $(G_1,\ell)$ and $(G_2,\ell)$ are $n$-completable, let $D_1$ and $D_2$ be the corresponding target matrices.
  We can find a system of future-directed null vectors $\x_u\in\mathbb{L}^n_+$ for $u\in V_1$ such that $(D_1)_{uv}=d_M(\x_u,\x_v)^2$ for $u,v\in V_1$,
  and $\y_u\in\mathbb{L}^n_+$ for $u\in V_2$ such that $(D_2)_{uv}=d_M(\y_u,\y_v)^2$ for $u,v\in V_2$.
  On the common clique, for all $u,v\in V_1\cap V_2$, we have $d_M(\x_u,\x_v)=d_M(\y_u,\y_v)=l(u,v)$.
  Therefore, there is a Lorentz transformation $\mathcal{L}$, such that $\mathcal{L}\x_u=\y_u$ for $u\in V_1\cap V_2$.
  Now we construct a system of vectors by setting $\mathbf z_u=\mathcal{L}\x_u$ for $u\in V_1$, and $\mathbf z_u=\y_u$ for $u\in V_2\setminus V_1$. 
  The matrix $D_{uv}=d_M(\mathbf z_u,\mathbf z_v)$ is a target matrix for $(G,\ell)$, therefore $\ell\in\mathcal{C}_K^n(G)$.
\end{proof}

This proves the ``if'' part of Theorem \ref{CompK}, since a chordal graph can be built up by clique-sums.
Therefore, for a choral graph, in order to tell if a length function is completable or not, we only need to check all the cliques.

\section{Distance geometry for spheres}\label{sec:Sphere}
We would like to mention some previous works that actually established a distance geometry for spheres.

We denote by $\mathbb{S}^n$ the set of $(n-1)$-spheres in $\hat{\mathbb{E}}^n$.
For a sphere $p$ in $\mathbb{S}^n$, we denote its center by $\c(p)$, and its radius by $r(p)$.
Therefore $(r(p),\c(p))$ is a point in the $(n+1)$-dimensional half-space $x_0>0$.

The following equation defines an analogue of the distance in Section \ref{sec:Dist}:
$$
d_S^2(p,q)=\frac{d_E^2(\c(p),\c(q))-r^2(p)-r^2(q)}{2r(p)r(q)}
$$
It seems to be first used by Darboux \cite{darboux1872}, 
and was referred to as ``separation'' in \cite{boyd1973}.
By abuse of language, we call $d_S$ a ``distance'' for spheres, even though it is not positive.
Related terms, such as ``distance space'', ``distance matrix'' and ``isometric embedding'', are also abused.

It is easy to verify that 
$d_S>1$ if the two spheres are externally disjoint, 
$=1$ if they are tangent from outside, 
$=-1$ if they are tangent from inside, 
$<-1$ if one is inside the other,
and $-1<d_S=-\cos\alpha<1$ if the two spheres intersect, where $\alpha$ is the angle of intersection,
so $d_S=0$ if they intersect orthogonally.

We have the following results:
\begin{thm}\label{EmbS}
  Let $(I,d)$ be a ``distance'' space \textup{(}symmetric but not necessarily non-negative\textup{)}.
  The following statements are equivalent
  \begin{enumerate}
      \renewcommand{\theenumi}{\rm (\roman{enumi})}
      \renewcommand{\labelenumi}{\theenumi}
    \item\label{sphere1} $(I,d)$ is ``isometrically'' embeddable into $(\mathbb{S}^n,d_S)$.
    \item\label{sphere2} $(-1)^{|J|}\det D_J\leq 0$ for all $J\subseteq I$, and the rank of $D_I$ is at most $n+2$.
    \item\label{sphere3} $D_I$ has exactly one positive eigenvalue, and at most $n+1$ negative eigenvalues.
  \end{enumerate}
\end{thm}
The equivalence between \ref{sphere1} and \ref{sphere2} seems to be first proved in \cite{darboux1872} for dimension $2$ and $3$.
A proof of the equivalence between \ref{sphere1} and \ref{sphere3} was sketched in \cite{boyd1973}.

We now show that Theorem \ref{EmbS} is intuitive in the framework of M\"obius geometry. 
There is a conventional way \cites{hertrich-jeromin2003, cecil2008, maxwell1981} to represent a spheres in $\mathbb{S}^n$ 
by vectors on the future directed one-sheet hyperboloid $$\mathbb{H}^{n+1}_+=\{\x\mid\langle\x,\x\rangle_{n+1,1}=1,t>0\}$$
Explicitely, a sphere $(r,\c)$ is represented by the vector
$$
\frac{1}{2r}\left(1-\|\c\|_2^2+r^2,2\c,1+\|\c\|_2^2-r^2\right)
$$
One can verify that this is an ``isometry'' between $(\mathbb{S}^n,d_S)$ and the Minkowski space $(\mathbb{H}^{n+1}_+,d_M)$.

This ``isometric'' embedding was used in \cite{maxwell1981} for studying sphere packings.
It is invariant under Lorentz transformations, 
reflecting the fact that combinatorics of spheres are M\"obius invariant.

With this representation, we can derive the result in Section \ref{sec:LCone} in another way:
Let $p$ be a sphere represented by a vector $\x_p$ sucht that $\langle\x_p,\x_p\rangle_{n+1,1}=1$.
a sphere tangent to $p$ is represented by vector such that $\langle\x,\x\rangle_{n+1,1}=1$
and $\langle\x,\x_p\rangle_{n+1,1}=-1$.
We then have $\langle\x+\x_p,\x+\x_p\rangle_{n+1,1}=0$ which defines an $n$-dimensional cone.
In fact, $\frac{\sqrt{2}}{2}(\x+\x_p)$ is an embedding of spheres kissing $p$ onto a lightcone as described in Section \ref{sec:LCone}.

\section*{Acknowledgement}
I would like to thank G\"unter M. Ziegler, my supervisor of PhD, for suggestions and careful review, 
thank Raman Sanyal for references, and thank Louis Theran and Karim Adiprasito for very helpful discussions. 

\bibliography{References}

\end{document}